\documentclass[12pt, a4paper, draft]{article}

\usepackage{amsmath, amsfonts, amsthm, amssymb, titlefoot}

\renewcommand{\amssubjname}{\textit{AMS 2000 subject classifications. }}

\newcommand     {\sym}[1]       {\operatorname{#1}}

\newcommand     {\F}            {{\mathbb F}}
\newcommand     {\Q}            {{\mathbb Q}}
\newcommand     {\Pro}          {{\mathbb P}}

\newcommand     {\Z}            {{\mathbb Z}}
\newcommand     {\Zloc}[1][(\ell)]{{\mathbb Z}_{#1}}
\newcommand     {\SL}[1][\Z]    {\operatorname{SL}(2,{#1})}

\newcommand     {\latdef}[2]    {\left(#1,#2\right)}
\newcommand     {\lat}[1][L]    {\underline{#1}}
\newcommand     {\dual}[1]      {{#1}^\sharp}
\newcommand     {\pseries}[1]   {[\hspace{-2pt}[{#1}]\hspace{-2pt}]}
\newcommand     {\legendre}[2]  {\big(\frac{#1}{#2}\big)}

\theoremstyle{plain}
\newtheorem*{Main Theorem}{Main Theorem}

\newtheorem{Theorem}{Theorem}
\newtheorem{Corollary}{Corollary}
\newtheorem*{Corollary-wl}{Corollary}
\newtheorem{Lemma}{Lemma}

\title{%
  Reduction mod $\ell$ of Theta Series of Level $\ell^n$
}

\author{%
  Nils-Peter Skoruppa}

\date{%
}


\amssubj{%
11F11   
11F27   
11F33   
}

\begin{document}

\maketitle

\begin{abstract}
  \noindent
  It is proved that the theta series of an even lattice whose level is
  a power of a prime $\ell$ is congruent modulo $\ell$ to an elliptic
  modular form of level~1. The proof uses arithmetic and algebraic
  properties of lattices rather than methods from the theory of
  modular forms.  The methods presented here may therefore be
  especially pleasing to those working in the theory of quadratic
  forms, and they admit generalizations to more general types of theta
  series as they occur e.g.  in the theory of Siegel or Hilbert
  modular forms.
\end{abstract}

\section{Statement of Results}

Let $\ell$ be a prime. We assume throughout that $\ell\ge 5$.  It is
well-known that every modular form of level $\ell^n$ is congruent
modulo $\ell$ to a modular form of level one
\cite[Th\'eor\`eme~5.4]{Serre}. This fact applies in particular to
theta series associated to quadratic forms whose level is a power of
$\ell$.  The purpose of this note is to prove a slightly more precise
statement and to discuss various consequences. Though the main result
is actually a statement about modular forms, the proof presented here
works only for theta series.  The virtue of this method of proof,
however, is that it admits generalizations to more general types of
theta series. We shall pursue this elsewhere. In this article we shall
prove the following theorem.

\begin{Main Theorem}
  Let $\lat = \latdef Lb$ be an even integral lattice whose level is a
  power of $\ell$, and let $e(\lat)$ be the sum of the elementary
  divisors of~$\lat$.  Then there exists a modular form~$f$ of level
  $1$ and weight $e(\lat)/2$ and with integral Fourier coefficients
  such that
  \[
  \theta_{\lat}:=\sum_{x\in L} q^{\frac12 b(x,x)} \equiv f \bmod \ell
  \]
\end{Main Theorem}

Here we are using standard terminology. By a lattice $\lat = \latdef
Lb$, we understand a free $\Z$-module $L$ of finite rank equipped with
a symmetric positive definite bilinear form $b$. We call it integral
if $b(x,x)$ is an integer for all $x$ in $L$, and we call it even, if
$b(x,x)$ is an even integer for all $x$ in $L$.  Note that in this
article the word lattice refers always to what is sometimes called
more precisely {\it positive definite lattice}.  The elementary
divisors of an even $\lat$ of rank $r$ are the $r$ elementary divisors
of the Gram matrix $G=\big(b(x_i,x_j)\big)_{i,j}$, where the $x_i$ run
through a $\Z$-basis of~$L$, and the level of $\lat$ is the smallest
natural number $l$ such that $lG^{-1}$ is an integral matrix with even
integers on its diagonal.  Of course, the elementary divisors and the
level do not depend on the particular choice of the $x_i$.

The congruence stated in the theorem has to be understood in the naive
sense that the difference of the series on both sides of the
congruence, viewed as formal power series in $q$ with coefficients in
$\Z$, lies in $\ell \Z\pseries q$. Here, as usual, modular forms as
functions of a variable~$z$ in the complex upper half plane are
identified with the formal power series obtained by expanding them in
powers of~$q=\exp(2\pi i z)$.

Note that $e(\lat)$, for an even $\lat$ as in the main theorem, is
divisible by~$4$.  In fact, the rank $r$ of the underlying $\Z$-module
$L$ is even since the determinant $d=\det(G)$ of its associated Gram
matrix $G$ is odd.  Moreover, using, for any integer $n\ge 0$, the
congruence $\ell^n\equiv 1+n(\ell-1) \bmod 2(\ell-1)$ and the fact
that $d$ equals the product of the elementary divisors of $\lat$, one
finds that
\[
\frac {e(\lat)}2 \equiv
\begin{cases}
  \frac r2 \bmod \ell-1            & \text{if~$d$ is a perfect square},\\
  \frac {r + \ell-1}2 \bmod \ell-1 & \text{otherwise.}
\end{cases}
\]
But $(-1)^{\frac r2}d\equiv 1 \bmod 4$, and hence $\frac r2$ is even
unless $d$ is not a perfect square and $d \equiv \ell \equiv -1 \bmod
4$.

The simplest examples for the main theorem are provided by binary
quadratic forms.  If $[a,b,c]$ denotes a positive definite integral
binary form (in Gauss notation) of discriminant $-\ell = b^2-4ac$ then
by the theorem
\[
\theta_{[a,b,c]} = \sum_{x,y\in\Z} q^{ax^2+bxy+cz^2}
\]
is congruent modulo~$\ell$ to a modular form of level 1 and weight
$\frac{\ell +1}2$. Noteworthy examples are
\begin{align*}
  \theta_{[1,1,2]} &\equiv E_4
  \equiv 1 + 2q + 4q^{2} + \cdots \bmod 7\\
  \theta_{[2,1,3]} &\equiv E_4^3-720 \Delta
  \equiv 1 + 2q^{2} + 2q^{3} + 2q^{4} + \cdots   \bmod 23\\
  \theta_{[2,1,4]} &\equiv E_4^4-960 E_4 \Delta
  \equiv 1 + 2q^{2} + 2q^{4} + \cdots   \bmod 31\\
  \theta_{[3,1,4]} &\equiv E_4^6-1440 E_4^3 \Delta+125280 \Delta^2
  \equiv 1 + 2q^{3} + 2q^{4} + \cdots   \bmod 47\\
  \theta_{[4,3,5]} &\equiv E_4^9-2160 E_4^6 \Delta+965520 E_4^3
  \Delta^2-27302400 \Delta^3
  \\&\qquad \equiv 1 + 2q^{4} + 2q^{5} + \cdots  \bmod 71\\
\end{align*}
Here and in the following, for an even positive integer $k$, we use
\begin{gather*}
  E_k = 1 - \frac {2k}{B_{k}}\sum_{n\ge 1}\sigma_{k-1}(n)\,q^n,\\
  \Delta = \frac {E_4^3-E_6^2}{12^3} = q\prod_{n\ge 1}(1-q^n)^{24} ,
\end{gather*}
with the Bernoulli numbers $B_2=\frac16$, $B_4=-\frac 1{30}$,
$B_6=\frac 1{42}$, \dots.  Note that the modular forms on the right
are the {\it extremal modular forms} of the respective weights,
i.e.~the modular forms $f_k$ of weight $k$ (here divisible by 4) whose
Fourier expansion is of the form $f_k \equiv 1 \bmod q^{\lfloor\frac
  k{12}\rfloor+1}$.  It is well known that for $2k=8,24,32,48$, these
extremal modular forms are equal to the theta series of even
unimodular lattices.  An obvious explanation for a congruence modulo
$\ell$ between two theta series associated to lattices $\lat$ and
$\lat [M]$ is the existence of an automorphism $\sigma$ of $\lat$
whose order is a power of $\ell$ and such that $\lat [M]$ is
isomorphic to the {\it fixed lattice} $\lat^\sigma = \latdef
{L^\sigma}{b'}$, where $L^\sigma$ is the submodule of all $x$ in $L$
which are fixed by $\sigma$ and where we use $b'$ for the restriction
of $b$ to $L^\sigma\times L^\sigma$
(cf.~Theorem~\ref{thm-automorphism-congruence} below).  And indeed, it
is known \cite{N-Sl} that the even unimodular lattices $E_8$, the
Leech lattice, $\Lambda_{RM}$ and $P_{48q}$, whose theta series are
equal to $f_4$, $f_{12}$, $f_{16}$ and $f_{24}$, have automorphisms of
order $7$, $23$, $31$ and $47$, respectively. (However, some of the
other lattices which have theta series equal to $f_{16}$ or $f_{24}$
do not have such automorphisms). Though the congruence for
$\theta_{[4,3,5]}$ does not prove that an extremal lattice of
dimension 72, if it existed, would have an automorphism of order $71$,
it supports such a speculation. There are exactly 55475 even
unimodular lattices of dimension 72 which have an automorphism of
order~71\footnote{The even 72-dimensional lattices having an
  automorphism of order 71 can be downloaded from
  http:/\kern-3pt/data.countnumber.de. A report on the computation of
  these lattices will be published elsewhere.}

We discuss some consequences of the main theorem. For this let
$\Theta(\ell^\infty)$ be the $\Zloc$-algebra generated by the theta
series $\theta_{\lat}$, where $\lat$ runs through all even
lattices whose level is a power of $\ell$.  Here and in the sequel we
use $\Zloc$ for the localization of $\Z$ at $\ell$, i.e.~for the ring
of rational numbers of the form $\frac rs$ with integers $r$, $s$ and
$s$ not divisible by $\ell$.  We have a natural filtration given by
the subalgebras $\Theta(\ell^n)$ generated by those $\theta_{\lat}$,
where the level of $\lat$ divides $\ell^n$.  Moreover, let $M_k$ be
the $\Zloc$-module of modular forms of level 1 and weight $k$ whose
Fourier coefficients are in $\Zloc$, and let $M$ be the
$\Zloc$-algebra generated by all these modular forms. Then $M$ is the
direct sum of the $M_k$, and $M=\Zloc[(\ell)] [E_4,E_6]$.  If
$\F_\ell=\Z/\ell\Z$ denotes the field with~$\ell$ elements we have a
natural map
\[
\Zloc\pseries q \rightarrow \F_\ell\pseries q \cong \Zloc\pseries
q/\ell \Zloc\pseries q, \quad f\mapsto \widetilde f
\]
which is defined by reducing each coefficient of $f$ modulo $\ell$.
Identifying modular forms and theta series with power series in $q$ we
can therefore rewrite the statement of the main theorem in the
(weaker) form
\[
\widetilde{\Theta(\ell^n)} \subseteq \widetilde M.
\]

If $\ell=5$ then $\widetilde {E_4} = 1$, and by the main theorem
$\widetilde {E_6} = \widetilde {\theta_{\lat}}$ for every quaternary
lattice $\lat$ of level~$5$ and determinant $25$. One may e.g.~take
the lattice $\lat [F]$ defined by the quaternary form
\[
F=\left(\begin{smallmatrix} 2&1&1&1\\1&2&0&1\\1&0&4&2\\1&1&2&4
  \end{smallmatrix}\right)
.
\]
We thus find
\[
\widetilde{M} = \widetilde{\Theta(5^\infty)} =
\F_\ell[\widetilde{\theta_{\lat [F]}}] .
\]

Similarly, if $\ell=7$, then $\widetilde{E_6} = 1$ and, by the main
theorem, $\widetilde{E_4} = \widetilde{\theta_{[1,2,8]}}$.  We
conclude
\[
\widetilde{M} = \widetilde{\Theta(7^{\infty})} = \F_\ell[\widetilde{
  \theta_{[1,1,2]}}] .
\]

If $\ell=11$, then $\widetilde{\theta_{[1,1,3]}} = \widetilde{E_6}$
Since $\widetilde{E_4}\widetilde{E_6} = \widetilde{E_{10}} = 1$ we
find here
\[
\widetilde{M} = \widetilde{\Theta(11^{\infty})} = \F_\ell[\widetilde{
  \theta_{[1,1,3]}}, 1/\widetilde{\theta_{[1,1,3]}}]
\]

More generally, it is not hard to deduce from the main theorem:
\begin{Corollary}
  \label{cor-filtration}
  In the notations of the preceding paragraphs one has
  \begin{align*}
    \widetilde{\Theta(1)} = \widetilde{\Theta(\ell)} =
    \widetilde{\Theta(\ell^2)} = \dots =
    \widetilde{\Theta(\ell^{\infty})} = \widetilde M \quad&\text{for
    }\ell \equiv 3 \bmod 4,\\
    \widetilde{\Theta(1)} \subsetneq \widetilde{\Theta(\ell)} =
    \widetilde{\Theta(\ell^2)} = \dots =
    \widetilde{\Theta(\ell^{\infty})} = \widetilde M \quad&\text{for
    }\ell \equiv 1 \bmod 4.
  \end{align*}
  In particular, $\widetilde{\Theta(\ell^{\infty})}$ is a finitely
  generated algebra over $\F_\ell$ of transcendence degree 1.
\end{Corollary}

\begin{Corollary}
  \label{cor-grading}
  $\widetilde{\Theta(\ell^{\infty})}$ is a $\Z/(\ell-1)\Z$-graded
  algebra:
  \[
  \widetilde{\Theta(\ell^{\infty})} = \bigoplus_{t \bmod \ell - 1}
  \widetilde{\Theta(\ell^{\infty})}^t,
  \]
  where $\widetilde{\Theta(\ell^{\infty})}^t$ is the
  $\F_\ell$-subspace generated by all $\widetilde{\theta_F}$ with
  $\frac {e(F)}2\equiv t \bmod \ell-1$.
\end{Corollary}

\begin{proof}[Proof of Corollaries~\ref{cor-filtration} and~\ref{cor-grading}]
  It is well-known that $M=\Zloc [(\ell)][E_4,E_6]$ and that
  $\widetilde M$ is isomorphic to $\F_\ell[X,Y]/(A-1)$ via the map
  $p(X,Y)\mapsto p(\widetilde{E_4},\widetilde{E_6})$, where $A$
  denotes the polynomial such that $E_{\ell-1}=A(E_4, E_6)$ (see
  \cite[Theorem~2]{Sw-D}).  Moreover, $E_4=\theta_{E_8}$, where $E_8$
  is the unique irreducible root lattice of dimension~$8$, in
  particular, $E_4$ is in $\Theta(1)$.

  For the proof of Corollary~\ref{cor-filtration} it thus suffices to
  show that (i)~$\widetilde{E_6}$ is in $\widetilde{\Theta(1)}$ if $\ell
  \equiv 3\bmod 4$, and that, for $\ell \equiv 1\bmod 4$,
  (ii)~$\widetilde{E_6}$ is in $\widetilde{\Theta(\ell)}$, and
  (iii)~there exists a $\theta$ in $\widetilde{\Theta(\ell)}$, which
  is not in $\widetilde{\Theta(1)}$.

  Using the fact that every (positive definite) even unimodular
  lattice has rank divisible by $8$, that $E_4^{k - l}\Delta^l$ ($0\le
  l \le \lfloor \frac k3 \rfloor$) is a $\Zloc$-basis of $M_{4k}$, and
  that, for the theta series $\theta_{\text{Leech}}$
  associated to Leech's lattice , we have
  $\theta_{\text{Leech}} = E_4^3-720 \Delta$, we find
  \[
  \Theta(1) = \bigoplus_{k\ge0} M_{4k} = \Zloc
  [(\ell)][\theta_{\text{Leech}}, \theta_{E_8}]
  \]
  (provided $\ell$ does not divide $720=2^4 \cdot 3^2 \cdot 5$).

  From this (i) follows immediately since $\widetilde{E_6} =
  \widetilde{E_6E_{\ell-1}}$ is in $\widetilde{M_{\ell+5}}$, and
  since, for $\ell \equiv 3 \bmod 4$, we have $M_{\ell+5}\subseteq
  \Theta(1)$.

  For (iii) we use another result of
  Swinnerton-Dyer~\cite[Theorem~2]{Sw-D}, namely
  \[
  \widetilde M = \bigoplus_{t\bmod \ell-1} \widetilde M^t ,
  \]
  where $\widetilde M^t$ is the sum of all $\widetilde {M_k}$ with
  $k\equiv t\bmod \ell-1$.  Now, if $\lat$ is an even rank~4 lattice
  of level $\ell$ and determinant $\ell^2$, the series $\theta_{\lat}$
  is in $\Theta(\ell)$ and, by the main theorem,
  $\widetilde\theta_{\lat}$ is in
  $\widetilde{M}^{l+1}=\widetilde{M}^2$. But then
  $\widetilde\theta_{\lat}$ is not in $\widetilde{\Theta(1)}$ since,
  by the preceding decompositions of $\Theta(1)$ and $\widetilde M$,
  the space $\widetilde{\Theta(1)}$, for~$\ell\equiv 1 \bmod 4$,
  equals the sum of those $\widetilde M^t$ where $t$ is divisible by
  $4$.

  By the main theorem $\widetilde{\Theta(\ell^\infty)}^t$ is contained
  in $\widetilde M^t$. Corollary~\ref{cor-grading} follows therefore
  from the decomposition of $\widetilde M$ of in the preceding
  paragraph.

  The proof of (ii) is more difficult. Let $\lat$ be an even lattice
  of rank 12 with level $\ell$ and whose determinant is a perfect
  square $\ge \ell^4$, say, equal to $\ell^{2n}$ (one may take the
  threefold direct sum of a suitable even quaternary lattice).  Then
  $\theta_{\lat}$ is a modular form of weight $6$ on $\Gamma_0(\ell)$
  with trivial character. We may therefore consider its trace
  \[
  \theta(z):=\sum_{A\in
    \Gamma_0(\ell)\backslash\SL}\theta_{\lat}(Az)(cz+d)^{-6}
  =\theta_{\lat}(z) + \sum_{t\bmod
    \ell}\theta_{\lat}\big(-1/(z+t)\big)\,(z+t)^{-6} ,
  \]
  which is a modular form of level 1, and equals hence a multiple of
  $E_6$.  Applying Poisson's summation formula to obtain
  \[
  \theta_{\lat}(-1/z)\,z^{-6} = -\ell^{-n}\sum_{x\in\dual L}e^{\pi i
    z b(x,x)} ,
  \]
  one finds
  \[
  \theta = \theta_{\lat} - \ell^{1-n}\sum_{\begin{subarray}{c} x\in\dual L\\
      b(x,x) \in \Z\end{subarray} } q^{\frac12 b(x,x)} ,
  \]
  in particular, $ \theta = (1-\ell^{1-n})E_6$. Here $\dual L$ denotes
  the set of all $y$ in $\Q\otimes L$ such that $b(y,x)$ is integral
  for all $x$ in $L$ (and where of course, $b$ has to be bilinearly
  extended to $\Q\otimes L$). From this we deduce
  \[
  E_6 \equiv \sum_{\begin{subarray}{c} x\in\dual L\\
      b(x,x) \in \Z\end{subarray} } q^{\frac12 b(x,x)} \bmod \ell .
  \]
  But the right hand side can be rewritten as
  \[
  \sum_{\begin{subarray}{c}u\in\Pro(\dual L/L)\\ \underline b(u,u)=0
    \end{subarray}
  } \theta_{\lat_u} - \left(\left|\{u\in\Pro(\dual L/L):\underline
      b(u,u)=0\}\right| - 1\right) \theta_{\lat} ,
  \]
  where $\Pro(\dual L/L)$ denotes the set of 1-dimensional subspaces
  of the $\F_\ell$-vector space $\dual L/L$, where $\underline b:
  \dual L/L\times \dual L/L \rightarrow \Q/\Z$ denotes the bilinear
  form induced by $b$, and where, for $u$ in $\Pro(\dual L/L)$, we use
  $\lat_u$ for the lattice with underlying module $\{x\in\dual L: x+L
  \in u\}$ and the corresponding restriction of $b$ as bilinear form.
  Note that $\lat_u$, for $\underline b(u,u)=0$, is an even integral
  lattice of level $\ell$ (here we use $\ell\not=2$). We conclude that
  $\widetilde{E_6}$ is indeed an element of
  $\widetilde{\Theta(\ell)}$.
\end{proof}

There is a final, almost trivial consequence of the main theorem which
might be noteworthy. Namely, if $\lat = \latdef Lb$ is an even lattice
and $\sigma$ an automorphism of $\lat$, then we may consider the fixed
lattice $\lat^\sigma$. It is easy to see that $\theta_{\lat}$ and
$\theta_{\lat^\sigma}$ are congruent modulo $\ell$ if the order of
$\sigma$ is a power of $\ell$
(cf.~Theorem~\ref{thm-automorphism-congruence} below). If,
furthermore, the level of $\lat^\sigma$ is a power of $\ell$ then we
may apply the main theorem to conclude that
$\widetilde{\theta_{\lat}}$ is the reduction modulo~$\ell$ of a a
modular form~$f$ of level~1.  (For a discussion of the level of
$\lat^\sigma$ in general see Lemma~\ref{lem-index-of-fixed-lattice} in
section~\ref{sec:lattices-and -automorphism}).  By the discussion
following the main theorem we know that the weight $k$ of $f$ is
congruent modulo $\frac{\ell-1}2$ to $\frac r2$, where $r$ is the rank
of $\lat^\sigma$, and that $r$ is even. The characteristic polynomial
of $\sigma$ is of the form
$(t-1)^r\phi_{l^{n_1}}(t)\dots\phi_{l^{n_t}}(t)$ (where~$\phi_h$ is
the $h$-th cyclotomic polynomial), and hence the rank $n$ of $\lat$ is
congruent modulo $\ell-1$ to $r$. In particular, $n$ is even.  We have
therefore proved:
 
\begin{Corollary}
  Let $\lat$ be an even lattice of rank $n$ which possesses an
  automorphism $\sigma$ such that its order and the level of the fixed
  lattice~$\lat^\sigma$ are powers of~$\ell$. Then there exists a
  modular form of level 1, weight $k\equiv \frac n2 \bmod
  \frac{\ell-1}2$ with integral Fourier coefficients such that
  $\theta_{\lat}\equiv f \bmod \ell$.
\end{Corollary}

\section{Proof of the Main Theorem}
\label{sec:lattices-and -automorphism}

The proof of the main theorem is suggested by two observations, which
we formulate here as Theorems~\ref{thm-automorphism-congruence}
and~\ref{thm-theta-representation-mod-l}. The first theorem is
well-known (however, we do not know any precise reference).

\begin{Theorem}
  \label{thm-automorphism-congruence}
  Let $\lat$ be an even integral lattice which possesses an
  automorphism~$\sigma$ whose order is a power of~$\ell$, and let
  $\lat^\sigma$ be the sublattice of elements fixed by~$\sigma$.  Then
  \[
  \theta_{\lat} \equiv \theta_{\lat^\sigma} \bmod \ell .
  \]
\end{Theorem}

\begin{proof}
  For a nonnegative integer $n$ let $X$ and $X^\sigma$ denote the set
  of all $x$ in~$\lat$ respectively $x$ in $\lat^\sigma$ such that
  $b(x,x)=2n$, where $b$ is the bilinear form of $\lat$. We have to
  show $|X| \equiv |X^\sigma| \bmod \ell$. But this is an immediate
  consequence of the orbit formula
  \[
  |X| = \sum_{x} [\langle \sigma\rangle:\sym{Stab}(x)] .
  \]
  Here $x$ runs through a complete set of representatives for the
  orbits in~$\langle\sigma\rangle\backslash X$ and $\sym{Stab}(x)$
  denotes the subgroup of elements in~$\langle \sigma\rangle$ fixing
  $x$.
\end{proof}

The second theorem concerns the Weil representation of an even lattice
with automorphism of $\ell$-power order.  For a given even lattice
$\lat = \latdef Lb$ of level $s$ and rank~$2k$ we let
$O_{\lat}=\Z[\zeta,1/\chi_{\lat}]$. Here $\zeta$ is a primitive $s$-th
root of unity and
\[
\chi_{\lat} = \sum_{\rho\in \sym{Det}(L)}\exp(2\pi i \, \underline
q(\rho)) ,
\]
where $\sym{Det}(\lat)=\dual L/L$ is the {\it determinant module
  of}~$\lat$, and where we use~$\underline q$ for the map (finite
quadratic form) $\underline q:\sym{Det}(\lat) \rightarrow \Q/\Z$
induced by $x\mapsto~\frac12 b(x,x)$.  Thus $O_{\lat}$ is a subring of
the cyclotomic field~$\Q(\zeta)$. Note that one has
$\chi_{\lat}=e^{\pi i k/2}|\sym{Det}(\lat)|^{\frac 12}$ (this identity
is sometimes called Milgram's theorem). We let~$W_{\lat}$ be the
$O_{\lat}$-submodule of $O_{\lat}\pseries{q^{\frac1s}}$ spanned by the
series $\theta_\rho := \sum_{x\in \rho}q^{\frac12 b(x,x)}$, where
$\rho$ runs through $\sym{Det}(\lat)$. It is well-known~\cite{Kl} that
$(\theta,A)\mapsto \theta|_{k}A$ defines a right action of~$\SL$
on~$W_{\lat}$ (provided~$k$ is integral).  Here we view the elements
of $W_{\lat}$ as functions of a variable~$z$ in the complex upper half
plane by setting $q=\exp(2\pi iz)$, and we use
$\big(f|_k\left(\begin{smallmatrix}a&b\\c&d\end{smallmatrix}\right)\big)(z)
= f\big(\frac{az+b}{cz+d}\big)\,(cz+d)^{-k}$.

Finally, if $\sigma$ denotes an automorphism of the (even) lattice
$\lat$, then, by linear extension, $\sigma$ acts naturally on
$\Q\otimes_\Z L$ and on $\sym{Det}(\lat)$.  We then have

\begin{Theorem}
  \label{thm-theta-representation-mod-l}
  Let $\lat$ be an even lattice of rank $2k$ which possesses an
  automorphism whose order is a power of $\ell$. Suppose that
  $\sym{Det}(\lat)^\sigma = 0$. Then~$k$ is even integral, and one has
  \[
  \theta_{\lat}|_{k}A \equiv \theta_{\lat} \bmod \ell W_{\lat}
  \]
  for all $A$ in $\SL$.
\end{Theorem}

\begin{proof}
  The action of $\SL$ on $W_{\lat}$ induces an action on the quotient
  $W_{\lat}/\ell W_{\lat}$, and the theorem states that $\theta_{\lat}
  + \ell W_{\lat}$ is invariant under this action.  It suffices to
  show this invariance for the generators
  $T=\left(\begin{smallmatrix}1&1\\0&1
    \end{smallmatrix}\right)$ and
  $S=\left(\begin{smallmatrix}0&-1\\1&0
    \end{smallmatrix}\right)$ of $\SL$.  The invariance under $T$
  is trivial. For showing the invariance under $S$ we use the formula
  \[
  \theta_{\lat}|_k S = \chi_{\lat}^{-1}\sum_{ \rho \in
    \sym{Det}(\lat)} \theta_{\rho} .
  \]
  (This formula follows from Poisson's summation formula,
  see~\cite{Kl} for details.)  Under this action of $\sigma$ on the
  determinant group $\sym{Det}(\lat)$ we have $\theta_{\sigma(\rho)} =
  \theta_\rho$.  Hence we can rewrite the preceding identity in the
  form
  \[
  \theta_{\lat}|_k S = \chi_{\lat}^{-1}\sum_{\rho} [\langle \sigma
  \rangle : \sym{Stab}(\rho)]\,\theta_\rho ,
  \]
  where~$\rho$ runs here through a complete set of representatives for
  the orbits in~$\langle \sigma \rangle \backslash \sym{Det}(\lat)$,
  and where, for each such $\rho$, we use $\sym{Stab}(\rho)$ for its
  stabilizer in~$\langle \sigma \rangle$. Similarly, we have
  \[
  \chi_{\lat}=\sum_{\rho} [\langle \sigma \rangle :
  \sym{Stab}(\rho)]\,e^{2\pi i \, \underline q(\rho)} .
  \]
  The theorem follows now from the fact that $0$ is the only element
  in $\sym{Det}(\lat)$ fixed by $\sigma$.

  Note that we have in particular proved $\chi_{\lat} \equiv 1 \bmod
  \ell\Z[\zeta]$, and the same argument implies $|
  \sym{Det}(\lat)|\equiv 1 \bmod \ell$. On the other hand, we have
  $\chi_{\lat}^2=e^{\pi i k}|\sym{Det}(\lat)|$.  We thus recognize
  that the rank $2k$ of $\lat$ must indeed be divisible by $4$ as
  claimed.
\end{proof}

The idea of proof of the main theorem is now apparent. Given a lattice
of $\ell$-power level we construct a lattice $\widehat L$ and an
automorphism $\sigma$ of $\ell$-power order such that $L$ is
isomorphic to the fixed lattice $\widehat\lat^\sigma$. Accordingly to
Theorem~\ref{thm-theta-representation-mod-l} one might expect then
that the theta series of $\widehat\lat$ is congruent modulo~$\ell$ to
a modular form of level 1, provided some additional assumptions on
$\widetilde \lat$ and the automorphism $\sigma$ hold true. Following
this idea we can indeed find a proof of the main theorem.  We postpone
the proof of the following theorem, which relies on a purely algebraic
property of quadratic forms, to the Appendix.

\begin{Theorem}
  \label{thm-embedding-as-fixed-module}
  For every even lattice $\lat$ whose level is a power of~$\ell$ there
  exists an even lattice $\widehat\lat$ which possesses an
  automorphism $\sigma$ of $\ell$-power order such that the sublattice
  of $\widehat\lat$ fixed by $\sigma$ is isomorphic to $\lat$. The
  lattice $\widehat \lat$ can be chosen so that its rank equals
  $e(\lat)$ and such that its level is not divisible by~$\ell$ or any
  prime $p \equiv -1 \bmod \ell$.
\end{Theorem}

Finally, we still need a lemma which assures that a lattice $\widehat
L$ as in the preceding theorem satisfies the hypothesis
$\sym{Det}(\widehat\lat)^\sigma=0$ of
Theorem~\ref{thm-theta-representation-mod-l}.

\begin{Lemma}
  \label{lem-index-of-fixed-lattice}
  Let $\sigma$ be an automorphism of $\lat = \latdef Lb$ whose order
  is a power of~$\ell$.  There are canonical embeddings of $\big(\dual
  L\big)^\sigma/L^\sigma$ into $\sym{Det}(\lat^\sigma)$ and
  $\sym{Det}(\lat)^\sigma$. The images under these embeddings are
  subgroups whose index is a power of~$\ell$, respectively.  In
  particular, if $\ell$ does not divide the determinant of $\lat$,
  then $\big(\dual L\big)^\sigma/L^\sigma$ can be identified with
  $\sym{Det}(\lat)^\sigma$.
\end{Lemma}
\begin{proof}
  Let $\ell^n$ denote the order of $\sigma$. We set $V=\Q\otimes_\Z L$
  and extend $b$ to a bilinear form on $V$. For a finitely generated
  $\Z$-submodule $M$ of $V$ we use $M^*$ for the set of $y$ in $\Q M$
  such that $b(y,M)\subset \Z$. (We have of course $L^*=\dual L$ with
  $\dual L$ as already used before.) Then $\sym{Det}(\lat^\sigma)$ can
  be identified with $(L^\sigma)^*/L^\sigma$. The natural embeddings
  of the theorem are given by the inclusion of $\big(\dual
  L\big)^\sigma/L^\sigma$ in $(L^\sigma)^*/L^\sigma$ and by the
  natural map $x + L^\sigma \mapsto x + L$.

  If $y$ is in $\big(L^\sigma\big)^*$ then $\sigma(y)=y$, hence
  $\ell^n y = s(y), $ where $s =\sum_{\tau\in\langle \sigma\rangle}
  \tau$. But
  \[
  b\big(s(y), L\big) = b\big(y , s(L)\big) \subseteq b\big(y,
  L^\sigma\big)\subseteq \Z .
  \]
  We conclude that~$\ell^n y$ is in $\big(\dual L\big)^\sigma$.

  Similarly, if $y + L$ is in $\sym{Det}(\lat)^\sigma$, then $\ell^n y
  \equiv s(y) \bmod L$, but $s(y)$ is in~$\big(\dual L\big)^\sigma$.
\end{proof}

\begin{proof}[Proof of the Main Theorem]
  Given a lattice $\lat$ whose level is a power of~$\ell$ we choose a
  lattice $\widehat\lat$ of rank~$2k$ and level~$s$ equipped with an
  automorphism~$\sigma$ as in
  Theorem~\ref{thm-embedding-as-fixed-module}. We choose
  $\widehat\lat$ such that $2k=e(\lat)$ and $s$ is not divisible
  by~$\ell$.  By Theorem~\ref{thm-automorphism-congruence} the
  series~$\theta_{\lat}$ is congruent
  to~$\theta:=\theta_{\widehat\lat}$ modulo~$\ell$.  Since~$\ell$ does
  not divide the determinant of $\widehat\lat$ and since the
  determinant of $\lat$ is a power of $\ell$ we conclude from the
  Lemma~\ref{lem-index-of-fixed-lattice} that $\sym{Det}(\lat)^\sigma
  =0$.  By Theorem~\ref{thm-theta-representation-mod-l} $k$ is even
  and we have $\theta|_kA \equiv \theta \bmod \ell
  O_{\lat}\pseries{q^{1/s}}$ for all $A\in\Gamma$.  It is well-known
  that $\theta$ is a modular form on $\Gamma_0(s)$ with a real
  character.  Because of the last congruence the character is trivial.
  The form $g:=\sum_A \theta|_k A$, with $A$ running through a
  complete set of representatives for $\Gamma_0(s)\backslash\Gamma$,
  is thus a modular form on~$\Gamma=\SL$.  But $g \equiv n \theta
  \bmod \ell O_{\lat}\pseries{q}$, where $n$ denotes the index
  of~$\Gamma_0(s)$ in $\Gamma$. Note that~$n=s\prod_{p|s}(1+\frac1p)$.

  If we write $g$ in the form $g=\sum c_{a,b} E_4^a\Delta^b$ or
  $g=\sum c_{a,b} E_4^aE_6\Delta^b$ (with $a,b$ running over all
  nonnegative integers such that $4a+12b=k$ in the first sum and
  $4a+12b=k-6$ in the second sum), we see that the coefficients
  $c_{a,b}$ are in~$O_{\lat}$, and that they are in fact congruent
  modulo~$\ell O_{\lat}$ to rational integers (since $g$ is congruent
  modulo $\ell O_{\lat}$ to $n\theta$).  Replacing the $c_{a,b}$ by
  these integers we can assume that $g$ has coefficients in $\Z$. But
  then $g \equiv n \theta \bmod \ell \Z \pseries{q}$ (since
  $\Z[\frac1\ell]\cap\ell O_{\lat} = \Z$).

  If we finally choose $\widehat\lat$ such that $s$ does not contain
  any primes congruent to~$-1$ modulo $\ell$, then $n$ is invertible
  modulo $\ell$ and we have proved the theorem.
\end{proof}

\section*{Appendix}

In this section we prove Theorem~\ref{thm-embedding-as-fixed-module}.
We shall say that a lattice $\lat = \latdef Lb$ can be {\it
  diagonalized over a} subring $R$ of $\Q$ if $R\otimes_\Z L$ contains
an orthogonal $R$-basis, i.e.~an $R$-basis $x_i$ such that
$b(x_i,x_j)=0$ for all $i\not= j$. (Here and in the following we use
the same letter $b$ for the bilinear extension of $b$ to $R\otimes L$
as for $b$ itself.)  It is easy to see that every lattice can be
diagonalized over~$\Zloc$.

\begin{Lemma}
  \label{lem-embedding-as-fixed-module}
  Let $\lat$ be an even lattice whose level is a power of $\ell$.
  Assume that $R$ is a localization of $\Z$ contained in $\Zloc$ such
  that $\lat$ can be diagonalized over~$R$. Then there exists a
  lattice $\widehat\lat$ which possesses an automorphism $\sigma$ of
  $\ell$-power order such that the sublattice of $\widehat\lat$ fixed
  by $\sigma$ is isomorphic to $\lat$. The lattice $\widehat \lat$ can
  be chosen so that its rank equals $e(\lat)$ and such that its level
  is a unit in $R$.
\end{Lemma}

\begin{proof}
  Let $\lat = \latdef Lb$, and let $e_i$ ($1\le i \le n$) be an
  orthogonal $R$-basis of $R\otimes_\Z L$. If $a_i$ is a $\Z$-basis of
  $L$ then $(a_i)_i = (e_i)_iM$ with a matrix~$M$ in~$\sym{GL}(n,R)$.
  Multiplying $M$ by the l.c.m.~$N$ of the denominators of its entries
  (which is a unit in $R$) and replacing $e_i$ by $e_i/N$ we can
  assume that~$L$ is contained in $H:=\bigoplus \Z e_i$. The index
  $[H:L]$ is an element of the group of units $R^*$ of $R$.  We can
  therefore find a natural number $d$ in $R^*$ such that $dH\subseteq
  L$ and such that $d \cdot b(x,x)$ is an even integer for all $x$ in
  $H$.  Write $b(e_i,e_i) = a_i\ell^{\alpha_i}$ with $a_i$ in $R^*$
  and an an integer $\alpha_i\ge 0$.  Note that the $\ell^\alpha_i$
  are the elementary divisors of $\lat$.  Denote by $\widehat {\lat
    [H]} = \latdef {\widehat H}c$ a lattice of rank $e(\lat)$ which
  possesses an orthonormal basis $e_{i,j}$ ($1\le i\le n$, $1\le j \le
  \ell^{\alpha_i}$) such that $c(e_{i,j},e_{i,j})=a_i$, and let
  $\sigma$ be the automorphism of $\widehat {\lat [H]}$, which, for
  each $i$, acts as
  \[
  e_{i,1}\mapsto e_{i,2} \mapsto \cdots \mapsto e_{i,\ell^{\alpha_i}}
  \mapsto e_{i,1} .
  \]
  The order of $\sigma$ is clearly a power of $\ell$.

  Finally, let $\widehat \lat$ be the sublattice of $\widehat {\lat
    [H]}$ whose underlying $\Z$-module is the set of all
  $\sum_{i,j}x_{i,j}e_{i,j}$ such that
  \[
  x_{i,1}\equiv x_{i,2} \equiv \cdots \equiv x_{i,\ell^{\alpha_i}}
  \bmod d
  \]
  for all $i$ and such that $\sum_i x_{i,1}e_i$ is in $L$.  We leave
  it to the reader to verify that $\widehat \lat$ is even, that its
  level is a unit in $R$, and that $\lat^\sigma$ is isomorphic
  to~$\lat$.
\end{proof}

The Theorem~\ref{thm-embedding-as-fixed-module} is now an immediate
consequence of the preceding lemma and the following theorem, whose
proof, however, seems to need some deeper facts from algebraic number
theory.

\begin{Theorem}
  \label{thm-diagonalization-over-localizations-of-Z}
  Let $S$ be the set of all nonzero integers which contain only primes
  $p\not=\ell$ and $p\not\equiv -1\bmod \ell$ as prime factors, and
  let $S^{-1}\Z$ the localization of $\Z$ at~$S$ (i.e.~the set of
  rational numbers $\frac rs$ with $r\in\Z$ and $s\in S$).  Then every
  lattice~$\lat$ can be diagonalized over $S^{-1}\Z$.
\end{Theorem}

\begin{proof}
  Set $R:=S^{-1}\Z$.  It suffices to show that every integral
  $R$-lattice $\lat [M] = \latdef Mb$ contains an $x$ such that
  $b(x,x)$ divides $b(y,z)$ (in $R$) for all $y$ and $z$ in $M$. Here
  by integral $R$-lattice $\latdef Mb$ we mean a free $R$-module of
  finite rank equipped with a (positive definite) symmetric bilinear
  map $b:M\times M \rightarrow R$.

  In fact, if this holds true, and if $\lat [M] = \latdef Mb$ is an
  integral $R$-lattice then choose an element $x_1$ in $M$ such that
  $b(x_1,x_1)$ divides all values of $b$ on $M\times M$ and let $M_1$
  be the orthogonal complement of $x_1$. Then $M=Rx_1 + M_1$ since,
  for any $y$ in $M$, the number $t:=b(x_1,y)/b(x_1,x_1)$ is in $R$
  and $y-tx_1$ is perpendicular to $x_1$. i.e.~$y-tx_1$ is in $M_1$.
  Replacing $\lat [M]$ by $\latdef {M_1}b$ we recognize that our claim
  follows by induction on the rank of $\lat [M]$.

  So let $\lat [M] = \latdef Mb$ be a $R$-lattice, and let $Rd$ be the
  ideal generated by all values of $b$ on $M\times M$.  Note that $Rd$
  coincides with the ideal generated by all $b(x,x)$ with $x$ in $M$
  (since $2b(x,y)=b(x+y,x+y)-b(x,x)-b(y,y)$ and $2$, for $\ell\ge 5$
  is a unit in $R$).  We want to show the existence of an $x$ in $M$
  such that $b(x,x)/d$ is a unit it $R$.

  If $M$ has rank 1 this assumption is trivial.  If the rank of $\lat
  [M]$ is greater than or equal to 2 we can proceed as follows.
  Choose a $y$ such that $b(y,y)\not=0$. We can then find a $z$ in $M$
  such that $Rb(y,y)+Rb(z,z)=dR$.

  Namely, for each prime $p$ dividing $b(y,y)/d$ (in $R$) which is not
  a unit in $R$ there is a $y_p$ in $M$ such that $p$ does not divide
  $b(y_p,y_p)/d$ (since $Rd$ is generated by all values $b(y,y)$).
  Using the Chinese Remainder Theorem we find a $z$ in $M$ such that
  $b(z,z)\equiv b(y_p,y_p) \bmod p$ for all $p$ in question, in
  particular, such that $b(x,x)/d$ and $b(z,z)/d$ are relatively
  prime.

  Finally, choose a unit $e$ in $R$ such that $Q(s,t) := \frac ed
  b(sy+tz,sy+tz)$ is a positive definite primitive binary quadratic
  form with integer coefficients. It suffices now to show that $Q$
  represents an integer not containing $\ell$ or a prime $p\equiv -1
  \bmod \ell$.  But this is assured by the subsequent
  Theorem~\ref{thm-ray-class-norms}.
\end{proof}

\begin{Theorem}
  \label{thm-ray-class-norms}
  Let $Q(x,y)$ be an integral primitive positive definite binary
  quadratic form, and let $\ell$ be a prime, $\ell \ge 5$. Then there
  exist integers $x,y$ such that $Q(x,y)$ is only divisible by primes
  $p\not\equiv 0,-1 \bmod \ell$.
\end{Theorem}

Note that the theorem does clearly not hold true for $\ell=2$. For
$\ell=3$ it does not hold true either: the quadratic form $2x^2+3y^2$
represents only numbers $n\equiv 0,-1 \bmod 3$ and each such $n$
contains at least one prime divisor $p\not\equiv +1 \bmod 3$.

\begin{proof}[Proof of Theorem~\ref{thm-ray-class-norms}]
  Let $Q(x,y)=ax^2+bxy+cy^2$, and write $b^2-4ac=Df^2$, where $D$ is a
  fundamental discriminant.

  Let $K=Q(\sqrt D)$, let ${\mathfrak O}=\Z+\Z\omega$ and ${\mathfrak
    O}_f = \Z+\Z f\omega$, where $\omega = \frac {D+\sqrt D}2$, and
  let $M=\Z a + \Z \frac{b+f\sqrt D}2$. Then $N_Q:=\{\frac
  1a\sym{N}(\alpha): \alpha \in M\}$ is the set of integers
  represented by $Q$ (we use $\sym{N}$ for the norm function on
  numbers or ideals in $K$). Moreover, $M{\mathfrak O}_f = M$.
  Replacing $Q$ by an equivalent form, if necessary, we may assume
  that $a$ and $\ell f$ are relatively prime (since we can find
  integers $x$ and $y$ such that $Q(x,y)$ is relatively prime to $\ell
  f$, e.g.~one may take for $x$ the product of all primes in $\ell f$
  dividing $a$ but not $c$, and for $y$ one may take the product of
  all primes in $\ell f$ not dividing $a$).

  But then $M+\ell f {\mathfrak O}_f = {\mathfrak O}_f$, which in turn
  implies that
  \[
  N:=\{\frac 1a {\sym{N}(\alpha)}: \alpha \in M{\mathfrak
    O}\cap(1+\ell f \mathfrak O)\}
  \]
  is a subset of $N_Q$.  Indeed, using $\ell f \mathfrak O \subseteq
  {\mathfrak O}_f$, we have
  \begin{align*}
    M{\mathfrak O}\cap(1+\ell f \mathfrak O) &\subseteq M{\mathfrak
      O}\cap{\mathfrak O}_f
    =\big(M{\mathfrak O}\cap{\mathfrak O}_f\big){\mathfrak O}_f\\
    &= \big(M{\mathfrak O}\cap{\mathfrak O}_f\big)\big(M+\ell f
    {\mathfrak O}_f\big) \subseteq M+\ell f M{\mathfrak O}\subseteq M
    .
  \end{align*}
  
  Now $M{\mathfrak O}=\Z a + \Z\omega$ (since $a$ and $f$ are
  relatively prime), hence $a = \sym{N}(M{\mathfrak O})$.  Therefore
  $N$ equals the set of norms of all integral ideals in the ideal
  class $\big({M{\mathfrak O}}\big)^{-1} P\in I/P$, where $P$ is the
  group of (fractional) ideals generated by the integral principal
  ideals $(\alpha)$ of $K$ such that $\alpha \equiv 1 \bmod \ell f$,
  and where $I$ is the group of fractional ideals of $K$ generated by
  all integral ideals relatively prime to~$\ell f$ (i.e.~$I/P$ is what
  is usually called the ray class group modulo $\ell f$).

  It remains to show that every ideal class $A$ in $C=I/P$ contains an
  integral ideal whose norm is in the group of units $(S^{-1}\Z)^*$,
  where $R:=S^{-1}\Z$ is the ring introduced in
  Theorem~\ref{thm-diagonalization-over-localizations-of-Z}. For the
  moment, we denote the set of~$A$ containing such an ideal by
  $\Sigma$.  Note that $\Sigma$ is a subgroup. It is obviously closed
  under multiplication. Moreover, if ${\mathfrak a}$ is an integral
  ideal in a class $A$ in $\Sigma$ whose norm is in $R^*$, then
  $A^{-1}$ contains the integral ideal ${\mathfrak
    a}^{-1}\sym{N}({\mathfrak a})^{\varphi(f\ell)}$ (where $\varphi$
  denotes Euler's $\varphi$-function), whose norm is again in $R^*$.
  We shall use repeatedly that every ideal class in $C$ contains prime
  ideals of degree one (as follows e.g.~from~\cite[p.~318]{Hecke}).

  We distinguish two cases.

  Case 1: $D=-\ell$. Let ${\mathfrak p}$ be a prime ideal of degree
  one in a given ideal class $A$ in $C$. For $p=\sym{N}({\mathfrak
    p})$ we then have $\legendre p\ell = \legendre Dp = +1$. In
  particular, $p \not\equiv 0,-1 \bmod \ell$ (since $\ell =-D \equiv 3
  \bmod 4$).

  Case 2: $D$ contains a prime factor different from $\ell$. Here we
  consider the map ${\mathfrak a}\mapsto \legendre{\sym{N}(\mathfrak
    a)}\ell$, which induces a group homomorphism of $C$.  The
  kernel~$\Gamma$ of this homomorphism has index at most 2 in $C$.  In
  fact, it has index exactly equal to 2: choose a prime $p$ such that
  $\legendre Dp=+1$ and $\legendre p\ell=-1$ (this is possible by
  Dirichlet's theorem on arithmetic progressions and since~$D$
  contains a prime different from $\ell$). Then $p$ is the norm of a
  prime ideal which is not in $\Gamma$.

  If $\ell \equiv 3 \bmod 4$ then $\Gamma$ is contained in $\Sigma$.
  Indeed, if $A$ is in $\Gamma$, then any prime ideal of degree one in
  $A$ with norm, say, $q$ satisfies $q\not\equiv 0,-1 \bmod \ell$
  (since $\legendre q\ell=+1$).  But the group $\Sigma$ is strictly
  bigger than $\Gamma$ as can be seen by choosing the prime $p$ of the
  last paragraph such that $p\not\equiv -1 \bmod \ell$ (for fulfilling
  this and $\legendre p\ell=-1$ at the same time we need $\ell\ge 5$).
  Since the index of $\Gamma$ in $C$ is 2 we conclude that its index
  in $\Sigma$ is~2 and $\Sigma=C$.

  If $\ell \equiv 1 \bmod 4$ then $C\setminus\Gamma$ is in $\Sigma$ as
  can be seen by picking in a given class $A$ in $C\setminus\Gamma$ a
  prime ideal of degree 1. In fact, its norm $q$ is different from
  $\ell$ and satisfies $q \not\equiv -1 \bmod \ell$ (since $\legendre
  q\ell=-1$). Since $\Gamma\not= C$ the set $C\setminus\Gamma$ is a
  (the) nontrivial $\Gamma$ coset, which is contained in $\Sigma$, and
  we again conclude~$\Sigma = C$.

  This proves the theorem.
\end{proof}


\bigskip
\begin{flushleft}
  \small
  Nils-Peter Skoruppa\\
  Universit\"at Siegen\\
  Fachbereich Mathematik,\\
  Walter-Flex-Stra{\ss}e 3\\
  57068 Siegen, Germany\\
  email: nils.skoruppa@uni-siegen.de
\end{flushleft}

\end{document}